\documentclass[leqno,12pt]{amsart}
\usepackage{latexsym}
\usepackage{mathrsfs}
\usepackage{amsmath}
\usepackage{amssymb}
\usepackage[bookmarks]{hyperref}
\usepackage{pstricks}

\newtheorem{theorem}{Theorem}[section]
\newtheorem{lemma}[theorem]{Lemma}

\theoremstyle{definition}

\newtheorem{definition}[theorem]{Definition}

\allowdisplaybreaks[1]

\newcommand{\C}{\mathbb{C}}
\newcommand{\D}{\mathbb{D}}

\newcommand{\N}{\mathbb{N}}

\newcommand{\R}{\mathbb{R}}

\newcommand{\oD}{\overline \D}

\def\R{\mathbb R}

\def\C{\mathbb C\/}

\def\tA{\tilde A}

\def\qa{q^{-1}(\{a\})}
\def\qb{q^{-1}(\{b\})}
\font\tenscr=rsfs10 
\font\sevenscr=rsfs7 
\font\fivescr=rsfs5 
\skewchar\tenscr='177 \skewchar\sevenscr='177 \skewchar\fivescr='177
\newfam\scrfam \textfont\scrfam=\tenscr \scriptfont\scrfam=\sevenscr
\scriptscriptfont\scrfam=\fivescr
\def\scr{\fam\scrfam}

\def\scr{\fam\scrfam}

\def\scrK{{\scr {K}}}

\def\jred{}

\def\eps{\varepsilon}
\def\normX#1{\|#1\|_X}

\begin{document}

\title[Constructing Essential Uniform Algebras]{A General Method for Constructing Essential Uniform Algebras}
\author{J. F. Feinstein}
\address{School of Mathematical Sciences, University of Nottingham, University Park, Nottingham NG7 2RD, UK }
\email{Joel.Feinstein@nottingham.ac.uk}
\author{Alexander J. Izzo}
\address{Department of Mathematics and Statistics, Bowling Green State University, Bowling Green, OH 43403, USA}
\email{aizzo@math.bgsu.edu}



\begin{abstract}
A general method for constructing essential uniform algebras with prescribed properties is presented.  Using the method, the following examples are constructed: an essential, natural, regular uniform algebra on the closed unit disc; an essential, natural counterexample to the peak point conjecture on each manifold of dimension at least three; and an essential, natural uniform algebra on the unit sphere in $\C^3$  containing the ball algebra and invariant under the action of the 3-torus.  These examples show that a smoothness hypothesis in some results of Anderson and Izzo cannot be omitted.
\end{abstract}
\maketitle

\section{Introduction}

Let
$X$ be a {compact space}, and let $C(X)$ be the algebra of all continuous complex-valued functions on $X$ with the supremum norm
$ \|f\|_{X} = \sup\{ |f(x)| : x \in X \}$.  A \emph{uniform algebra} on $X$ is a closed subalgebra of $C(X)$ that contains the constant functions and separates
the points of $X$.
One can often obtain a uniform algebra $A$ with particular properties on a specific space $X$ by finding a uniform algebra $B$ with the desired properties on a subspace $E$ of $X$ and then taking $A$ to consist of those continuous functions on $X$ whose restrictions to $E$ belong to $B$.  However, sometimes it is of interest to know whether there are examples on the space $X$ that do not arise from algebras on a subspace in this trivial manner.
This issue is made precise using the 
notion of essential set and essential uniform algebra, as introduced by Bear in \cite{Bear} (see also \cite[pp. 144--147]{browder}).  The {\it essential set} $E$ for a uniform
algebra $A$ on a space $X$ is the smallest closed subset of $X$ such
that $A$ contains every continuous function on $X$ that vanishes on
$E$. Note that $A$ contains every continuous function whose
restriction to $E$ lies in the restriction of $A$ to $E$.  The
uniform algebra $A$ is said to be {\it essential} if $E=X$.
The reader may also wish to consult \cite{tomiyama}, where Tomiyama determined the connection between the essential set of $A$ and the antisymmetric decomposition of the maximal ideal space of $A$.

In this paper, we present a general method for constructing essential uniform 
algebras.  We then use the method to obtain three particular examples of essential 
uniform algebras with special properties.  These examples demonstrate contrasts 
between uniform algebras generated by smooth functions and those not generated by such functions.

{In \cite{dePaepe}, de Paepe gave a very different method for constructing essential uniform algebras with specified properties. 
However, de Paepe's algebras and the compact spaces they are defined on do not have the properties we require.}

We say that a uniform algebra $A$ on $X$ is {\it nontrivial\/} if $A\neq C(X)$, and is {\it natural\/} (on $X$)
if $X$ is the maximal ideal space of $A$ (under the usual identification of points of $X$ with evaluation functionals).  For the definitions of other terms used in this paper, see the next section.

Our main result is as follows.
\begin{theorem}\label{maintheorem}
Let $A$ be a nontrivial uniform algebra on a compact space $K$.  Let $X$ be a compact metric space every
{non-empty} open subset of which contains a nowhere dense subspace homeomorphic to $K$.  
Then there exists a sequence $\{K_n\}_{n=1}^\infty$ of pairwise disjoint, nowhere dense subspaces of $X$ each homeomorphic to $K$ such that $\bigcup_{n=1}^\infty K_n$ is dense in $X$ and ${\rm diam} (K_n)\rightarrow 0$.  If  homeomorphisms $h_n:K_n\rightarrow K$ are chosen and we set $A_n=\{f\circ h_n:f\in A\}$, then the collection of functions $\tA=\{f\in C(X): f|K_n\in A_n \hbox{\ for\ all\ $n$}\}$ is an essential uniform algebra on $X$.
The equality $\tA|K_n=A_n$ holds for all $n$.

Furthermore, the following relations hold between the properties of $A$ and the properties of $\tA$:
\item(i) $\tA$ is natural if and only if $A$ is natural;
\item(ii) $\tA$ is regular on $X$ if and only if $A$ is regular on $K$;
\item(iii) $\tA$ is normal if and only if $A$ is normal;
\item(iv) every point of $X$ is a peak point for $\tA$ if and only if every point of $K$ is a peak point for $A$;
\item(v) $\tA$ has bounded relative units if and only if $A$ has bounded relative units;
\item(vi) $\tA$ is strongly regular if and only if $A$ is strongly regular;
\item(vii) $\tA$ is a Ditkin algebra if and only if $A$ is a Ditkin algebra.

The uniform algebra $\tA$ has bounded relative units, with bound $1$, at every point of $X \setminus \bigcup_{n=1}^\infty K_n$, and (hence) each of these points is a peak point for $\tA$, at each of these points $\tA$ satisfies Ditkin's condition, and at each of these points $\tA$ is strongly regular.
\end{theorem}


One could also consider a more general situation in which, rather than starting with one space $K$ and uniform algebra $A$, one deals with a sequence of spaces $K_n$, possibly not all homeomorphic, with dense union in $X$, and algebras $A_n$ on the $K_n$.  We leave this to the interested reader.

Note that every compact metric space is homeomorphic to a nowhere dense subset of the Hilbert cube \cite[Theorem~V~4]{HW}. This allows us to apply Theorem \ref{maintheorem} whenever $A$ is a nontrivial uniform algebra on a metrizable compact space $K$ in order to construct an essential uniform algebra on the Hilbert cube sharing many properties with $A$. However we are mostly interested in examples on spaces whose topological dimension is finite.

Our original motivation came from the following question: Does there exist an essential, natural, regular uniform algebra on the closed unit disc?
As an application of our main theorem, we prove that the answer is affirmative.
To our knowledge, the first example of an essential, natural, regular uniform algebra on a locally connected compact metric space was given in \cite[Example 2.9]{feinheath2010}, using the algebra $R(K)$ for a suitable compact plane set $K$.  However the set $K$ obtained was a \lq classical' Swiss cheese set, and so (although both connected and locally connected), was infinitely connected.

\begin{theorem}\label{discexample}
There exists an essential, natural, regular uniform algebra on the closed unit disc $\oD=\{z\in \C: |z|\leq 1\}$.
\end{theorem}

However, applying a theorem of Michael Freeman \cite[Theorem~4.1]{freeman}, we shall see that no such uniform algebra can be generated by smooth functions.
(To say that a collection of functions $F$ generates $A$ means that the collection of all polynomials in the functions in $F$ is a dense subset of $A$.) 

\begin{theorem}\label{discnonexample}
There does not exist an essential, natural, regular uniform algebra generated by $C^1$-smooth functions on a compact two-dimen\-sional $C^1$-manifold-with-boundary.
\end{theorem}

A theorem of John Anderson and the Alexander Izzo \cite[Theorem~4.2]{AI2} classifies all the natural uniform algebras containing the identity function $z$ and generated by $C^1$-smooth functions on the closed disc.  
 The example constructed below in the proof of Theorem~\ref{discexample} is easily seen to contain the function $z$ and thus shows that this classification does not continue to hold without the smoothness hypothesis.

Our other two examples are related to the so called peak point conjecture.  This conjecture asserted that if a uniform algebra $A$ is natural on $X$ and every point of $X$ is a peak point for $A$, then $A=C(X)$.  (A point $x \in X$ is said to be a {\it peak point\/} for $A$ is there exists $f \in A$ with $f(x) = 1$ and $|f(y)| < 1$ for all $y \in X \setminus \{x\}$.)  A counterexample to this peak-point conjecture was produced by Brian Cole in 1968  \cite{Co} (or see \cite[Appendix]{browder}, or \cite[Chapter~3, Section 19]{St-book}), and other counterexamples have
been given since then  with a variety of additional properties, and including examples which are generated by smooth functions on manifolds (see, for example, \cite{Bas, FeinsteinStronglyRegular, Feinstein, FeinsteinMorris, izzo_counterexample, I2009, stout_book2}).
Nevertheless, Anderson and Izzo \cite{AI1, AI2, AI3} and Anderson, Izzo, and Wermer \cite{AIW1, AIW2} have established peak point theorems under certain smoothness hypotheses.  One of those results, which we state here, asserts that certain uniform algebras can never be essential.

\begin{theorem}[\cite{AI2}, Theorem 1.2]\label{andersonizzo}
Let $A$ be a uniform algebra on a compact
$C^1$-manifold-with-boundary $M$.  Assume that $A$ is generated by
$C^1$-smooth functions, that  $A$ is natural on $M$, and
that every point of $M$ is a peak point for $A$.  Then the uniform
algebra $A$ is not essential.  In fact, the essential set for $A$
has empty interior in $M$.
\end{theorem}

Our next example shows that this theorem becomes false if the hypothesis that the algebra is generated by $C^1$-smooth functions is dropped.

\begin{theorem}\label{peakpointexample}
On every compact manifold-with-boundary $X$ of dimension greater than or equal to 3 there exists an essential,
natural uniform algebra such that every point of $X$ is a peak point.
In addition, we may arrange for this uniform algebra to have bounded relative units.
\end{theorem}

Another theorem of Anderson and Izzo asserts that the peak point conjecture holds for uniform algebras generated by $C^1$-smooth functions on a compact manifold-with-boundary of dimension {\it two\/}.  It remains an open question whether the smoothness hypothesis can be dropped from that theorem.

For our final example, we modify so as to make essential, an example constructed by the second author \cite[Theorem~2.4]{I2009} in response to a question raised by Ronald Douglas in connection with his work on a conjecture of William Arveson in operator theory.  Before stating the result we recall some notions used in the statement.  

The
ball algebra $A(S)$ on the unit sphere $S\subset \C^n$ consists of the restrictions to
the sphere of the functions that are continuous on the closed unit
ball ${\overline B_n} \subset \C^n$ and holomorphic on the open unit
ball $B_n$. The action of the $n$-torus $T^n$ on $S$ is the map $T^n
\times S \rightarrow S$ given by $\bigl( (e^{i\theta_1}, \dots,
e^{i\theta_n}), (z_1, \dots, z_n) \bigr) \mapsto (e^{i\theta_1}z_1,
\dots, e^{i\theta_n}z_n)$. To say that an algebra $A$ on $S$ is
invariant under the action of the $n$-torus means that the function
$(z_1,\dots, z_n) \mapsto f(e^{i\theta_1}z_1, \dots,
e^{i\theta_n}z_n)$ is in $A$ for each point $(e^{i\theta_1}, \dots,
e^{i\theta_n})\in T^n$ whenever $f$ is in $A$.

\begin{theorem}\label{sphereexample}
There exists an essential, natural uniform algebra on the unit sphere $S$ in $\C^3$ that  contains the ball algebra and is invariant under the action of the 3-torus on $S$.
\end{theorem}

The example in \cite{I2009} is generated by $C^\infty$-smooth functions and satisfies the conditions above except that it is not essential.  Note that since every point of the sphere is a peak point for the ball algebra, Theorem~\ref{andersonizzo} shows that it is not possible for the algebra to be simultaneously both essential and generated by smooth functions.

Although the uniform algebras we construct are essential, they do include many non-constant, real-valued functions, and so they are not antisymmetric. This leaves open the question of whether or not there are antisymmetric uniform algebras with these properties.
For example, we do not know whether there exists an antisymmetric, natural, regular uniform algebra on the closed unit disc.

Theorem~1.1 is proved in Section~3.  In Section~4, we prove Theorems~1.2, 1.5, and~1.6.  Theorem~1.3 is proved in Section~5.

\section{Preliminaries}~\label{term} ~\label{prelim}
{Throughout the paper, by a \textit{compact space} we shall mean a non-empty, compact, Hausdorff topological space; by a \textit{compact plane set} we shall mean a {non-empty}, compact subset of the complex plane.}

We assume that the reader has some familiarity with uniform algebras. For further background we refer the reader to \cite{browder,Ga,St-book}. The reader may also consult \cite{Dales,Kaniuth} for the general theory of commutative Banach algebras, and more details concerning regularity conditions and their applications.

\begin{definition}
Let $A$  be a uniform algebra on a compact space $X$.
Recall that we say that $A$ is natural (on $X$) if the only non-zero multiplicative linear functionals on $A$ are given by evaluations at the points of $X$, in which case $X$ may be identified with the maximal ideal space of $A$.
We say that $A$ is \emph{regular on} $X$ if, for each closed subset $E \subseteq X$ and each $x \in X \setminus E$, there is a function $f \in A$ with $f(x)=1$ and $f(E) \subseteq \{0\}$; the algebra $A$ is \emph{normal on} $X$ if, for each pair of disjoint closed subsets $E$ and $F$ of $X$, there is a function $f \in A$ with $f(E) \subseteq \{0\}$ and $f(F)\subseteq \{1\}$.
We say that $A$ is \emph{regular} if it is natural and regular on $X$; $A$ is \emph{normal} if it is natural and normal on $X$.
\end{definition}

It is standard that the uniform algebra $A$ on $X$ is normal if and only it is natural and regular on $X$ \cite[Theorem 27.2]{St-book}, and that this holds if and only if $A$ is normal on $X$ \cite[Theorem 27.3]{St-book}.

The following elementary result was implicitly assumed throughout \cite{FeinsteinSomerset2000} (in the more general setting of Banach function algebras).
The proof involves an elementary compactness argument: the details are given in \cite[Proposition 2.1]{2012FeinsteinMortini}.

\begin{lemma}\label{reg}
A uniform algebra $A$ on $X$ is regular on $X$ if and only if for every pair of distinct points $x_0$ and $x_1$ in $X$ there is a function $f$ in $A$ such that $f$ is zero on a neighborhood of $x_0$ and $f(x_1)=1$.
\end{lemma}

The next lemma is proved using a similar elementary compactness argument: we leave the details to the reader.
\begin{lemma}\label{normal}
A uniform algebra $A$ on $X$ is normal (on $X$) if and only if for every pair of distinct points $x_0$ and $x_1$ in $X$ there is a function $f$ in $A$ such that $f$ is zero on a neighborhood of $x_0$ and one on a neighborhood of $x_1$.
\end{lemma}

We now introduce some important ideals, and recall some stronger regularity conditions.
\begin{definition}
Let $A$  be a uniform algebra on a compact space $X$.
Let $x \in X$. We define the ideals $M_x$, $J_x$ in $A$ as follows:
\begin{eqnarray*}
 M_x&=&\{f\in A:f(x)=0\}\,;\\
J_x&=&\left\{f\in A: f^{-1}(\{0\}) \textrm{ is a neighborhood of } x\right\}.
\end{eqnarray*}
We say that $A$ is {\it strongly regular at\/} $x$ if
$J_x$ is dense in $M_x$;
$A$ satisfies
\emph{Ditkin's condition at} $x$ if, for every $f\in M_x$ and every $\eps > 0$, there is $g\in J_x$ with $\normX{gf-f}<\eps$; for $C \geq 1$, $A$
has \emph{bounded relative units at} $x$ \emph{with bound} $C$ if,
for every compact set $F\subseteq X\setminus\{x\}$, there is $f\in J_x$ with
$\normX{f}\le C$, such that $f(F)\subseteq \{1\}.$
We say that $A$ \emph{has bounded relative units at} $x$ if there exists $C \geq 1$ such that $A$ has bounded relative units at $x$ with bound $C$.
The algebra $A$ is {\it strongly regular} if
it is strongly regular at
every point of $X$;
$A$ is a \emph{Ditkin algebra} if it satisfies Ditkin's condition at each $x\in X$; $A$ is a
\emph{strong Ditkin algebra} if it is strongly regular and, for every $x \in X$, the ideal $M_x$ has a bounded approximate identity; $A$ has \emph{bounded
relative units} if it has bounded relative units at each $x\in X$.
\end{definition}

The reader may find a short survey of the relationships between these regularity conditions for uniform algebras in \cite{feinheath2007}. In particular we note the following.
Let $A$  be a uniform algebra on a compact space $X$, and let $x \in X$. 
If $A$ has bounded relative units at $x$, then $A$ satisfies Ditkin's condition at $x$, $A$ is strongly regular at $x$, $M_x$  has a bounded approximate identity and, for all $C>1$, $A$ has bounded relative units at $x$ with bound $C$. Thus if $A$ has bounded relative units, then every $C>1$ serves as a global bound. (For general Banach function algebras the bound may genuinely depend on the point.) 
The uniform algebra $A$ has bounded relative units if and only if it is a strong Ditkin algebra and this implies that $A$ is a Ditkin algebra; if $A$ is a Ditkin algebra then $A$ is strongly regular; if $A$ is strongly regular then $A$ is natural and regular on $X$ and hence normal.

Now suppose that $X$ is metrizable. It is then standard that $M_x$ has a bounded approximate identity if and only if $x$ is a peak point for $A$. 
(See, for example, \cite[p.~101]{browder}, or \cite[Theorem 4.3.5]{Dales}, and note that strong boundary points coincide with peak points when $X$ is metrizable.)
Thus, if $A$ has bounded relative units at $x$, then $x$ must be a peak point for $A$. 
In this setting, $A$ has bounded relative units if and only if $A$ is strongly regular and every point of $X$ is a peak point.

Examples of non-trivial, strongly regular uniform algebras on compact metrizable spaces were given in \cite{FeinsteinStronglyRegular}, including some examples with bounded relative units.

We shall need some results about metrizability. The first of these is essentially \cite[Corollary 26.16]{Jameson}. We include a short proof for the reader's convenience.
\begin{lemma}\label{metrizable}
Let $X$
be a compact metrizable space and $Y$ be a quotient space of $X$.  If $Y$ is Hausdorff, then $Y$ is metrizable.
\end{lemma}

\begin{proof}
It is well known that a compact Hausdorff space $Z$ is metrizable if and only if $C(Z)$ is separable.
Let $q:X\rightarrow Y$ be a quotient map of $X$ onto $Y$.  The map $q^*:C(Y)\rightarrow C(X)$ given by $q^*(f)=f\circ q$ embeds $C(Y)$ isometrically into $C(X)$.  Thus the hypotheses of the lemma give that $C(Y)$ is isometric to a subspace of a separable metric space and hence is separable.  Consequently, $Y$ is metrizable.
\end{proof}

Let $(X,d)$ be a metric space, let $x \in X$, and let $r>0$. We denote by $B_r(x)$ the set 
$\{y \in X: d(y,x)<r\}$, which is called \emph{the open ball of radius $r$ about $x$}.

\begin{lemma}\label{quotient}
Let $(X,d)$ be a compact metric space, let $(K_n)$ be a sequence of pairwise disjoint, closed subsets of $X$
whose diameters go to zero, and let $Y$ be the quotient space obtained from $X$ by identifying each $K_n$ to a point.  Then $Y$ is metrizable.
\end{lemma}

\begin{proof}
By Lemma~\ref{metrizable}, it suffices to show that $Y$ is Hausdorff.
Denote the quotient map of $X$ onto $Y$ by $q$.  Let $a$ and $b$ be distinct points of $Y$.  Then $\qa$ and $\qb$ are disjoint closed sets in $X$.  Choose disjoint neighborhoods $U_a$ and $U_b$ of $\qa$ and $\qb$, respectively.  Define $\scrK_a$ to be the collection of those $K_n$ that intersect $U_a$ but are not contained in $U_a$ and define $\scrK_b$ in the same way but with $U_a$ replaced by $U_b$.  Set
$$V_a=U_a\setminus (\bigcup_{K\in \scrK_a} K) \quad {\rm and}\quad
V_b=U_b\setminus (\bigcup_{K\in \scrK_b} K).$$
Then $V_a$ and $V_b$ are disjoint sets containing $\qa$ and $\qb$, and each of $V_a$ and $V_b$ is saturated (i.e., $q^{-1}(q(V_a))=V_a$ and $q^{-1}(q(V_b))=V_b$).  Thus to complete the proof, it suffices to show that $V_a$ and $V_b$ are open.  Let $x\in V_a$ be arbitrary.  Since $U_a$ is open, there is an $r>0$ such that the open ball $B_r(a)$ of radius $r$ about $a$ is contained in $U_a$.  Since the diameters of the $K_n$ go to zero, there are at most finitely many $K_n$, say $K_{n_1},\ldots, K_{n_k}$, that intersect $B_{r/2}(a)$ and are not contained in $U_a$.  Then $B_{r/2}(a)\setminus (K_{n_1}\cup\cdots\cup K_{n_k})$ is an open set about $a$ contained in $V_a$.  Thus $V_a$ is open.  The same argument applies to $V_b$.
\end{proof}

\section{Proof of Theorem~\ref{maintheorem}}

\begin{proof}
[Proof of Theorem \ref{maintheorem}]
Since $X$ is a compact metric space, there is a countable dense subset $\{a_n\}_{n=1}^\infty$ in $X$.  Set $n_1=1$ and $r_1=1$.  By hypothesis there is a nowhere dense subspace $K_1$ of the ball $B_1(a_1)$ homeomorphic to $K$.  Let $n_2$ be the smallest positive integer such that $n_1<n_2$ and $a_{n_2}$ is not in $K_1$, choose $r_2$ such that $0<r_2<1/2$ and the ball $B_{r_2}(a_{n_2})$ is disjoint from $K_1$, and then choose a nowhere dense subspace $K_2$ of $B_{r_2}(a_{n_2})$ homeomorphic to $K$.  Now suppose we have chosen positive integers $n_1<n_2< \cdots < n_m$, positive numbers $r_1,\ldots, r_m$, and nowhere dense subspaces $K_1,\ldots, K_m$ of $X$.  By the Baire category theorem, the union $K_1\cup\cdots\cup K_m$ is nowhere dense in $X$.  Thus infinitely many points of  $\{a_n\}_{n=1}^\infty$ lie in the complement of $K_1\cup\cdots\cup K_m$.    Let $n_{m+1}$ be the smallest positive integer such that $n_m< n_{m+1}$ and $a_{n_{m+1}}$  is not in $K_1\cup\cdots\cup K_m$, choose $r_{m+1}$ such that $0<r_{m+1}<1/(m+1)$ and such that the ball $B_{r_{m+1}}(a_{n_{m+1}})$ is disjoint from  $K_1\cup\cdots\cup K_m$, and then choose a nowhere dense subspace $K_{m+1}$ of $B_{r_{m+1}}(a_{n_{m+1}})$ homeomorphic to $K$.  By induction, we then obtain a sequence $\{K_n\}_{n=1}^\infty$ as in the third sentence of the theorem.

Now suppose that homeomorphisms $h_n:K_n \rightarrow K$ are given, and define $A_n$ and $\tilde A$ as in the statement of the theorem.

We shall frequently use certain quotient spaces of $X$ in the remainder of the proof. Let $Y$ be the quotient space obtained from $X$ by
identifying each $K_n$ to a point, and let $q$ be the quotient map of $X$ onto $Y$.
Also, for each positive integer $m$, let $Y_m$ be the quotient space obtained from $X$ by identifying each $K_n$ with $n\neq m$ to a point, and let $q_m$ be the quotient map of $X$
onto $Y_m$. By Lemma~\ref{quotient}, $Y$ and all of the $Y_m$ are metrizable.

We now establish the equality $\tA|K_m=A_m$ for all $m$.  Fix a positive integer $m$.
We may identify $K_m$ with the subset $q_m(K_m)$ of $Y_m$.
Given an arbitrary function $g$ that belongs to $A_m$, apply the Tietze extension theorem to extend $g$ to a continuous complex-valued function $f$ on $Y_m$.  Then the function $f\circ q_m$ is  obviously in $\tilde A$ and $(f\circ q_m)|K_m=g$.

We next show that
$\tilde A$ is a uniform algebra.  All conditions are obvious except that $\tilde A$ separates points. For that let $a$ and $b$ be two distinct points of $X$.  If there is
a positive integer $m$ such that both $a$ and $b$ belong to $K_m$, then the equality $\tA|K_m=A_m$ gives at once a function in $\tA$ that separates $a$ and $b$.  Otherwise
we may use the quotient space $Y$ defined above: since in this case $q(a)\neq q(b)$, there is a continuous real-valued function $f$ on $Y$ that separates $q(a)$ and $q(b)$.  The function $f\circ q$ is then obviously in $\tilde A$ and separates $a$ from $b$.

To show that $\tilde A$ is essential, first note that by the Baire category theorem, the union of the $K_n$ has empty interior in $X$.  Thus letting $U$ be an arbitrary nonempty open subset of $X$, we know that there is a point $a$ in $U$ lying in none of the $K_n$.  Now choose $r>0$ such that the ball $B_{2r}(a)$ is contained in $U$.  Note that there must be infinitely many $K_n$ that intersect the ball $B_r(a)$.  Because ${\rm diam} (K_n)\rightarrow 0$, it must then be that for some positive integer $m$, the set $K_m$ is contained in $U$.  Because the restriction of each function in $\tilde A$ to $K_m$ lies in $A_m$ and $A$ is a nontrivial uniform algebra, it follows that $U$ must intersect the essential set for $\tilde A$. Consequently, $\tilde A$ is essential.

\def\jcsobs{our observation about closed sets}

At this point it is convenient to make an observation concerning the separation of certain pairs of closed sets, as this will help at a number of points in the remainder of the proof. Let $Y$ and $q$ be as above.
We claim that, for each pair of closed subsets $E$ and $F$ of $X$ such that
$q(E) \cap q(F) = \emptyset$,
there is a real-valued function $f \in \tA$ with $\normX{f}=1$ such that $f\equiv 0$ on some neighborhood of $E$ and $f\equiv 1$ on some neighborhood of $F$.
We shall refer to this below as \emph{\jcsobs}.
Note here that $q(E) \cap q(F) = \emptyset$ if and only if $E \cap F = \emptyset$ and there is no positive integer $m$ such that $E$ and $F$  both meet $K_m$.

To prove our claim, let $E$ and $F$ be such a pair of closed subsets of $X$.
Since $q(E)$ and $q(F)$ are disjoint closed subsets of $Y$, and $Y$ is normal, we may choose, in $Y$, a pair of disjoint closed neighborhoods $N_E$, $N_F$  of $q(E)$, $q(F)$ respectively. By Urysohn's lemma, there is a continuous real-valued function $g\in C(Y)$ with $\|g\|_Y=1$ and such that $g(N_E)\subseteq\{0\}$ and $g(N_F)) \subseteq \{1\}$. Set $f=g\circ q$. Then $f \in \tA$,
$\normX{f}=1$, $f\equiv 0$ on the neighborhood $q^{-1}(N_E)$ of $E$, and $f\equiv 1$ on the neighborhood $q^{-1}(N_F)$ of $F$. Thus $f$ has the desired properties.

We now establish the properties of the points of $X \setminus \bigcup_{n=1}^\infty K_n$ from the end of the statement of the theorem.
Let $x \in X \setminus \bigcup_{n=1}^\infty K_n$. The fact that $\tA$ has bounded 
relative units at $x$ with bound $1$ is immediate from \jcsobs, taking $E=\{x\}$
and considering an arbitrary compact subset
$F$ of $X \setminus \{x\}$.
It follows that $\tA$ satisfies Ditkin's condition at $x$, $\tA$ is strongly regular at $x$, and $x$ is a peak point for $\tA$. (The fact that $x$ is a peak point may also be proved
directly.)

We now turn to establishing the relations between the properties of $\tA$ and those of $A$.

{(i) It is well known that given a uniform algebra $B$ on a compact space $\Sigma$ and a closed subspace $E$ of $\Sigma$, the maximal ideal space of the uniform algebra $\overline{B|E}$ can be identified with a subspace of the maximal ideal space of $B$.  (This subspace is the {\it $B$-convex hull\/} of $E$.)  Furthermore, under this identification, the maximal ideal space of $B$ is the disjoint union of the maximal ideal spaces of its restrictions to its maximal sets of antisymmetry.  (See the remark after \cite[Theorem~II.13.2]{Ga}.)  It follows that a uniform algebra is natural if and only if each of its restrictions to a maximal set of antisymmetry is natural.}

Now let $Y$ be the quotient space defined above.   The continuous real-valued functions separate points on $Y$, and each of these functions induces a function on $X$ belonging to $\tA$.
It follows that each set of antisymmetry for $\tA$ is either contained in some $K_n$ or is a singleton.  Furthermore, it then follows from the equality $\tA|K_n=A_n$ that the sets of antisymmetry for $\tA$ contained in $K_n$ are precisely the sets of antisymmetry for $A_n$.  Consequently, the collection of maximal sets of antisymmetry for $\tA$ consists precisely of the maximal sets of antisymmetry for the $A_n$ and the singletons lying in no $K_n$.

{That $\tA$ is natural if and only if $A$ is natural follows at once from the conclusions of the preceding two paragraphs.}

(ii) The proof of this is similar to the proof of (iii) so we leave it as an exercise for the reader.

(iii) Suppose $A$ is normal.  Let $x_0$ and $x_1$ be distinct points of $X$.  We seek a function in $\tA$ that is zero on a neighborhood of $x_0$ and one on a neighborhood of $x_1$.  The case when $x_0$ and $x_1$ do not lie in a common $K_n$ is
easy, as the sets $E=\{x_0\}$ and $F=\{x_1\}$ then satisfy the conditions for \jcsobs.

 Assume now that for some $m$, both $x_0$ and $x_1$ belong to $K_m$.  Let $Y_m$ and $q_m$ be the  quotient space and quotient map defined above.
 Since by hypothesis $A$ is normal, there is a function $h$ in $A_m$ such that $h$ is zero on a neighborhood of $x_0$ in $K_m$ and one on a neighborhood of $x_1$ in $K_m$.  Now regard $K_m$ as a subspace of $Y_m$, and choose disjoint neighborhoods $V_0$ and $V_1$ of $x_0$ and $x_1$, respectively, in $Y_m$, such that $h=0$ on $V_0\cap K_m$ and $h=1$ on $V_1\cap K_m$.  Then choose neighborhoods $W_0$ and $W_1$ of $x_0$ and $x_1$ in $Y_m$ with $\overline W_0\subset V_0$ and $\overline W_1 \subset V_1$.  Now define $g$ on $K_m\cup \overline W_0 \cup \overline W_1$ by
$$g(x)=\begin{cases}&h(x)\ \mbox{for}\ x\in K_m\\
&0\ \mbox{for}\  x \in \overline W_0\\
&1\ \mbox{for}\  x \in \overline W_1\\
\end{cases}\\$$
Then $g$ is well-defined and continuous.  Apply the Tietze extension theorem to extend $g$ to a continuous function $h$ on $Y_m$.  Then the function $f=h\circ q_m$ is the function we seek.

That normality of $\tA$ implies normality of $A$ is trivial.

(iv) Suppose every point of $K$ is a peak point for $A$.  We are to show that an arbitrary point $x$ in $X$ is a peak point for $\tA$.  When $x$ lies in no $K_n$, we already know that $x$ is a peak point for $\tA$,
as noted earlier.

Consider now the case when $x$ belongs to $K_m$.  Let $Y_m$ and $q_m$ be as above.
  By hypothesis there is a function $g$ in $A_m$ that peaks at $x$.    By the Tietze extension theorem, $g$ extends to a continuous complex-valued function $h$ on $Y_m$ with supremum norm 1.  Let $\rho$ be a nonnegative continuous function on $Y_m$ that is identically equal to 1 on $K_m$ and strictly less than 1 everywhere else on $Y_m$.  Then 
   the function $(\rho h) \circ q_m$ belongs to $\tA$ and peaks at $x$.

That every point of $K$ is a peak point for $A$ if every point of $X$ is a peak point for $\tA$ is trivial.

(v)
As noted in Section \ref{prelim}, a uniform algebra on a compact metric space has bounded relative units if and only if it is strongly regular and every point is a peak point. Thus (v) may be deduced from (iv) and (vi). However, we include a direct proof.

As the property of having bounded relative units passes to restriction algebras, it follows that if $\tA$ has bounded relative units, then so does $A$.
Now suppose that $A$ has bounded relative units. We show that $\tA$ has bounded relative units. We already know that $\tA$ has bounded relative units at every point of $X \setminus \bigcup_{n=1}^\infty K_n$. Now let $m \in \N$, and let $x \in K_m$. Then $A_m$ has bounded relative units at $x$, with bound $C$, say. We show that $\tA$ also has bounded relative units at $x$ with bound $C$.

Let $E$ be a compact subset of $X \setminus \{x\}$, and set $F=E \cap K_m$ (which may be empty). Then there is a $g \in A_m$ with $||g||_{K_m} \leq C$ such that $g$ is constantly $0$ on some relatively open neighborhood $U$ of $x$ in $K_m$, and $g$ is constantly $1$ on $F$.  Let $Y_m$ and $q_m$ be as above. 
As before, we may identify $K_m$ with the subset $q_m(K_m)$ of $Y_m$, and then we have $x \in Y_m \setminus q_m(E)$.
Let $N$ be a compact neighborhood of $x$ in $Y_m\setminus q_m(E)$ such that $N \cap K_m \subseteq U$. Then there is an obvious continuous extension $h$ of $g$ to $K_m \cup N \cup q_m(E)$ such that $h \equiv 0$ on $N$ and $h \equiv 1$ on $q_m(E)$, and this extension $h$ still has uniform norm at most $C$. We may then apply the Tietze extension theorem to extend $h$ to all of $Y_m$ with the same norm: we also call this extension $h$. Set $f=h\circ q_m$. Then $f \in \tA$, $||f||_X \leq C$, $f$ is constantly $0$ on the neighborhood $q^{-1}(N)$ of $x$, and $f$ is constantly $1$ on $E$. This shows that $\tA$ has bounded relative units at $x$ with bound $C$.

(vi) As in (v), only one implication requires a proof, as the other is trivial.

Suppose that $A$ is strongly regular.
We already know that $\tA$ is strongly regular at every point of $X \setminus \bigcup_{n=1}^\infty K_n$.
Now let $m \in \N$, and let $x \in K_m$. Then $A_m$ is strongly regular at $x$. We show that $\tA$ is also strongly regular at $x$.
Let $f \in \tA$ with $f(x)=0$, and let $\eps>0$. Then there exists $g \in A_m$ with $g$ constantly $0$ on a relatively open neighborhood $U$ of $x$ in $K_m$ and such that
$|| f|_{K_m} - g||_{K_m} < \eps$.
Now let $Y_m$ and $q_m$ be as above. Regarding $K_m$ as a subset of $Y_m$, we may use the Tietze extension theorem as before to extend $g$ to a function $h \in C(Y_m)$ such that $h$ vanishes on a neighborhood $N$ of $x$ in $Y_m$. Set $\tilde g = h \circ q_m$. Then $\tilde g \in J_x$ in $\tA$, and $\tilde g|_{K_m}=g$, so $||f-\tilde g||_{K_m} < \eps$.

Set $F=K_m$ and set $E=\{x \in X: |f(x)-\tilde g(x)| \geq \eps\,\}\,.$
Then $E$ and $F$ satisfy the conditions for \jcsobs, so there is a function $a \in \tA$ with $\normX{a}=1$ such that $a \equiv 0$ on a neighborhood of $E$ and $a \equiv 1$ on a neighborhood of $F$.
We have $\normX{a(f-\tilde g)} < \eps$.
Set $b=(1-a)(f-\tilde g)$. Then $b\equiv 0$ on a neighborhood of $K_m$ and
$||b-(f-\tilde g)||_X = \normX{a(f-\tilde g)} < \eps$. Thus $b+\tilde g \in J_x$, and $||(b+\tilde g)-f||_X < \eps$.
This shows that (in $\tA$) $J_x$ is dense in $M_x$, as required.

(vii) This is similar to (vi), and we leave the details to the reader. We note only that the function $g$ from (vi) should now have the form $g_1 f|_{K_m}$ for some $g_1$ in $A_m$ which vanishes on a relatively open neighborhood of $x$ in $K_m$. This function $g_1$ should then be extended to a function $\tilde g_1 \in J_x$ in $\tA$, at which point we set $\tilde g = f \tilde g_1$. This ensures that the function $b+\tilde g$ from (vi) is in $f J_x$.
\end{proof}

\section{Construction of specific examples}
Recall that, for a compact plane set $K$, the uniform algebra $R(K)$
 is the uniform closure in $C(K)$ of the set of restrictions to $K$ of rational functions with no poles in $K$.

\begin{proof}[Proof of  Theorem \ref{discexample}]
Let $K$ be a compact set in the plane such that $R(K)$ is regular but nontrivial.  (An example of such a set was constructed by McKissick \cite{Mc} and is commonly referred to as McKissick's Swiss cheese.  McKissick's example is also presented in \cite[pp.~344--355]{St-book}, and a substantial simplification of part of the argument involved is given in \cite{Ko}.)  Set $A=R(K)$ and $X=\oD$.  The set $K$ is nowhere dense in the plane, and the uniform algebra $R(K)$ is natural.  Thus the uniform algebra $\tA$ furnished by Theorem~\ref{maintheorem} is essential, natural, and regular.
\end{proof}

\begin{proof}[Proof of  Theorem \ref{peakpointexample}]

Let $E$ be McKissick's Swiss cheese \cite{Mc} mentioned in the proof of Theorem~\ref{discexample} above.  Recall that $E$ is nowhere dense in the plane.  Let $K$ be the Cartesian product of $E$ and a circle.  Richard Basener \cite{Bas} (or see \cite[Example~19.8]{St-book}) produced a nontrivial, natural uniform algebra $A$ on $K$ with every point of $K$ a peak point.  Since a solid torus embeds in $\R^3$, the same is true of $K$, and because $E$ is nowhere dense in the plane, $K$ embeds in $\R^3$ so as to be nowhere dense.  
Thus every open subset of the manifold $X$ contains a nowhere dense subspace homeomorphic to $K$.  Now the uniform algebra $\tA$ furnished by Theorem~\ref{maintheorem} is essential and natural, and every point of $X$ is a peak point for $\tA$.

If we wish to ensure, in addition, that $\tA$ has bounded relative units, then instead of 
Basener's example, we may use the nontrivial uniform algebra with bounded relative units on a compact metric space $L$ constructed in \cite[Theorem 3.6]{FeinsteinStronglyRegular}, as long as we ensure that $L$ has topological dimension~1. 
This can be done by starting the construction from McKissick's example, as McKissick's Swiss cheese has topological dimension~1, and the construction preserves the topological dimension of the maximal ideal space.
The well-known embedding theorem for spaces of finite topological dimension due to Menger and N\" obeling (\cite[Theorem~V~2]{HW} or \cite[Theorem~50.5]{Mu}) then gives that $L$ embeds in $\R^3$.  Consequently every open subset of the manifold $X$ contains a nowhere dense subspace homeomorphic to $L$, and applying Theorem~\ref{maintheorem} with $L$ in place of $K$ yields the desired example.
\end{proof}

\begin{proof}[Proof of  Theorem \ref{sphereexample}]

This result does not follow directly from Theorem~\ref{maintheorem} because the requirement that the algebra be invariant under the action of the 3-torus precludes defining the algebra in terms of the behavior of functions on arbitrarily small sets.  Instead we combine ideas from the proof of Theorem~\ref{maintheorem} with the proof of \cite[Theorem~2.4]{I2009}.

Let $K$ be a compact set in the plane such that $R(K)$ is nontrivial but is such that the only Jensen measures for $R(K)$ are the point masses (for instance McKissick's Swiss cheese).
Then $K$ is nowhere dense in the plane.
Consider the open quarter disc 
$$Q=\{ a + bi \in \C: a > 0, b > 0, \,\, {\rm and} \,\, |a|^2 + |b|^2 < 1\}.$$  
The argument in the first paragraph of the proof of Theorem~\ref{maintheorem} can be repeated to obtain a sequence $\{K_n\}_{n=1}^\infty$ of disjoint, nowhere dense subspaces of $Q$ each homeomorphic to $K$ such that $\bigcup_{n=1}^\infty K_n$ is dense in $Q$ and ${\rm diam} (K_n)\rightarrow 0$.  For each $n$, define $X_n$ by
$$X_n = \bigl\{(z_1, z_2, z_3) \in S : |z_1| + |z_2|i \in K_n \bigr\}.$$
Then clearly each $X_n$ is a compact subset of $S$ that is invariant under the action of the $3$-torus.  Note that the functions $z_1, z_2, $ and $z_3$ have no \hbox{zeros} on $X_n$.  Let $\pi_n: X_n \rightarrow K_n$ be defined by $\pi_n(z_1, z_2, z_3) = |z_1| + |z_2|i$.  Now let $B_n$ be the uniform algebra on $X_n$ generated by the functions $z_1, z_2, z_3, z_1^{-1}, z_2^{-1}, z_3^{-1}$ and all functions of the form $g \circ \pi_n$ with $g \in R(K_n)$.  Then $B_n$ is clearly invariant under the action of the $3$-torus.  Finally let $A=\{f\in C(S): f|X_n\in B_n \hbox{\ for\ all\ $n$}\}$.  Then $A$ is invariant under the action of the $3$-torus.  Clearly $A$ contains the ball algebra $A(S)$.  Thus the proof will be complete once we show that $A$ is natural and essential.

The quotient space obtained from $\overline Q$
 by identifying each $K_n$ to a point is metrizable by Lemma~\ref{quotient},
 so the real-valued continuous functions on this quotient space separate points.  It follows that each set of antisymmetry for $A$ is either contained in some $X_n$ or is a singleton.  It is shown in the proof of \cite[Theorem~2.4]{I2009} that the algebras $B_n$ are natural.  That $A$ is also natural now follows from the fact, used above to prove Theorem~\ref{maintheorem}(i), that the maximal ideal space of a uniform algebra is the union of the maximal ideal spaces of its restrictions to its maximal sets of antisymmetry.

The proof that $A$ is essential is similar to the proof that $\tA$ is essential in Theorem~\ref{maintheorem}.  Let $\pi: S \rightarrow \overline Q$ be defined by $\pi(z_1, z_2, z_3) = |z_1| + |z_2|i$.   Because $A$ is invariant under the action of the
3-torus, the essential set for $A$ is also invariant under that action.  Consequently, to show that $A$ is essential, it suffices to show that for every nonempty open set $U$ of $Q$, the set $\pi^{-1}(U)$ intersects the essential set for $A$.
Let $U$ be an arbitrary nonempty open set of $Q$.  By the Baire category theorem we know that there is a point $p$ in $U$ lying in none of the $K_n$.  Now choose $r>0$ such that the ball $B_{2r}(p)$ is contained in $U$.  Note that there must be infinitely many $K_n$ that intersect the ball $B_r(p)$.  Because ${\rm diam} (K_n)\rightarrow 0$, it must then be that for some positive integer $m$, the set $K_m$ is contained in $U$.  Then $X_m=\pi^{-1}(K_m)$ is contained in $\pi^{-1}(U)$.  The restriction of each function in $A$ to $X_m$ lies in $B_m$, and it is shown in the proof of \cite[Theorem~2.4]{I2009} that $B_m$ is a nontrivial uniform algebra.  It follows that $\pi^{-1}(U)$ intersects the essential set for $A$.
\end{proof}

\section{Proof of Theorem \ref{discnonexample}}

\begin{proof}[Proof of  Theorem \ref{discnonexample}]
We suppose that there is such a uniform algebra $A$ on a compact
two-dimensional $C^1$-manifold-with-boundary $M$ and derive a contradiction.
Let $F$ be a collection of $C^1$-smooth functions that generate $A$.  
 Following Freeman \cite{freeman}, we define the {\it exceptional set\/} $E$ of $F$ by
$$E=\{p\in M: df(p)\wedge dg(p)=0\ \hbox{for all $f,g$ in $F$}\}.$$
Note that $E$  is  a closed set in $M$.  By \cite[Theorem~4.1]{freeman}, the essential set for $A$ is contained in the exceptional set.  Thus we must have  $E=M$.  Note that this says that at each point $p$ of $M$, the vector space of complex differentials spanned by the set $\{df(p):f\in F\}$ has dimension at most one.  Consequently, if there is a function $f$ in $F$ whose real and imaginary parts form a local coordinate system for $M$ on some open set $U$, then every function in $A$ is holomorphic on $U$ in the complex coordinate system given by $f$.  (See \cite[p.~43]{freeman}.)  This contradicts the regularity
of $A$, so there is no such function $f$ in $F$.  We conclude that for each point $p$ in $M$, the vector space of {\it real\/} differentials spanned by the set
$$\{du(p): \hbox{$u$ is the real or imaginary part of a function $f\in F$}\}$$
has dimension at most one.  But then the functions in $A$ fail to separate points on $M$, contrary to the definition of a uniform algebra.
\end{proof}

\subsection*{Acknowledgements}
The research in this paper was begun while the second author was visiting the United Kingdom supported by the London Mathematical Society (LMS~Scheme~2, grant~ref~21212). Progress was also made on the paper while the second author was a visitor at Indiana University.  He thanks the Department of Mathematics for its hospitality.

\end{document}